\documentclass[final,1p,times]{elsarticle}

\usepackage{amscd,amsfonts,amsmath,amssymb,amsthm,graphicx,mathrsfs,titletoc}

\newtheorem{theorem}{Theorem}

\newtheorem{lemma}[theorem]{Lemma}

\newtheorem{remark}[theorem]{Remark}

\numberwithin{equation}{section}
\numberwithin{theorem}{section}

\newcommand{\la}{\Delta}
\def\na{\nabla}

\renewcommand{\d}{\delta}

\newcommand{\ra}{\rightarrow}
\newcommand{\p}{\partial}
\newcommand{\f}{\frac}
\def\a{\alpha}
\def\lam{\lambda}

\def\e{\epsilon}

\newcommand{\be}{\begin{equation}}
\renewcommand{\ra}{\rightarrow}
\newcommand{\ee}{\end{equation}}
\newcommand{\bea}{\begin{eqnarray}}
\newcommand{\eea}{\end{eqnarray}}
\newcommand{\bna}{\begin{eqnarray*}}
\newcommand{\ena}{\end{eqnarray*}}

\journal{***}

\begin{document}

\begin{frontmatter}

\title{Another remark on a result of Ding-Jost-Li-Wang}

\author{Xiaobao Zhu}
  \ead{zhuxiaobao@ruc.edu.cn}
\address{ School of Mathematics,
Renmin University of China, Beijing 100872, P. R. China\\
Dedicated to Professor Jiayu Li on the occasion of his 60th birthday}

\begin{abstract}

Let $(M,g)$ be a compact Riemann surface, $h$ be a positive smooth function on $M$. It is well known that the functional
$$J(u)=\frac{1}{2}\int_M|\nabla u|^2dv_g+8\pi\int_M udv_g-8\pi\log\int_Mhe^{u}dv_g$$
achieves its minimum under Ding-Jost-Li-Wang condition. This result was generalized to nonnegative $h$
by Yang and the author. Later, Sun and Zhu (arXiv:2012.12840) showed the Ding-Jost-Li-Wang condition is also sufficient when $h$ changes sign,
which was reproved later by Wang and Yang (J. Funct. Anal. 282: Paper No. 109449, 2022)
and Li and Xu (Calc. Var. 61: Paper No. 143, 2022) respectively using a flow approach. The aim of this note is to give a new proof of Sun and Zhu's result. Our proof is based on the variational method and the maximum principle.
\end{abstract}

\begin{keyword}
Kazdan-Warner equation; blow-up analysis; sign-changing function; variational method; maximum principle
\MSC[2020] 46E35; 58C35; 35B33; 35B50
\end{keyword}

\end{frontmatter}

\titlecontents{section}[0mm]
                       {\vspace{.2\baselineskip}}%\bfseries}
                       {\thecontentslabel~\hspace{.5em}}
                        {}
                        {\dotfill\contentspage[{\makebox[0pt][r]{\thecontentspage}}]}
\titlecontents{subsection}[3mm]
                       {\vspace{.2\baselineskip}}%\bfseries}
                       {\thecontentslabel~\hspace{.5em}}
                        {}
                       {\dotfill\contentspage[{\makebox[0pt][r]{\thecontentspage}}]}

\setcounter{tocdepth}{2}
%\tableofcontents

%\setcounter{tocdepth}{1}
%\tableofcontents

%\tableofcontents
\section{Introduction}
Let $(M,g)$ be a compact Riemann surface, $h$ a smooth function on $M$. For simplicity, we assume in this paper that the area of $M$ equals one.
In 1974, Kazdan-Warner \cite{KW74} asked, under what kind of conditions on $h$, the equation
\begin{align}\label{kw-eq}
\Delta u=8\pi-8\pi he^u
\end{align}
has a solution. In literature, one calls it the Kazdan-Warner problem.  Easy to find, a necessary condition to the Kazdan-Warner problem
is $\max_Mh>0$.
Twenty three years later, Ding, Jost, Li and Wang \cite{DJLW97} attacked this problem successfully.
By using variational method and blow-up analysis, they gave a sufficient condition to the Kazdan-Warner problem.

In \cite{DJLW97}, the authors assumed $h>0$ and minimized the functional
\begin{align*}
J(u)=\f{1}{2}\int_M|\nabla u|^2dv_g+8\pi\int_M udv_g
\end{align*}
in the function space
\begin{align*}
H_1=\left\{u\in W^{1,2}(M): \int_M he^udv_g=1\right\}.
\end{align*}
Notice that critical points of $J$ in $H_1$ are solutions of equation (\ref{kw-eq}). The classical Moser-Trudinger inequality (cf. \cite{Fontana}, \cite{DJLW97})
\begin{align}\label{MT}
\log\int_M e^{u}dv_g\leq \frac{1}{16\pi}\int_M|\nabla u|^2dv_g+\int_M udv_g+C,
\end{align}
tells us that $J$ is bounded from below but not coercive. To see for what kind of $h$ the functional $J$ can achieve its minimum, they considered the perturbed functional
\begin{align}\label{pert}
J_\e(u)=\f{1}{2}\int_M|\nabla u|^2dv_g+(8\pi-\e)\int_M udv_g-(8\pi-\e)\log\int_Mhe^{u}dv_g
\end{align}
where $\e>0$. By the Moser-Trudinger inequality (\ref{MT}) one knows, $J_\e$ is coercive and achieves its minimum at some $u_\e\in H_1$ which satisfies
\begin{align}\label{kw-eq1}
\Delta u_\e=(8\pi-\e)-(8\pi-\e)he^{u_\e}.
\end{align}
If $u_\e$ is bounded in $W^{1,2}(M)$, then $u_\e$ converges to $u_0$ which minimizes $J$. Or else, $u_\e$ blows up, one has $\lambda_\e=\max_Mu_\e=u_\e(x_\e)\ra+\infty$ and
$x_\e\ra p$ as $\e\ra0$. (In this note, we do not distinguish sequence and its subsequence.) Denote $r_\e=e^{-\lam_\e/2}$. They showed the following facts.

\textbf{Fact 1.}  $u_\e(x_\e+r_\e x)-\lambda_\e\ra-2\log(1+\pi h(p)|x|^2)$ in $C_{\text{loc}}^{1}(\mathbb{R}^2)$ as $\e\ra0$.

\textbf{Fact 2.} $u_\e-\overline{u_\e}\ra G_p(x)$ weakly in $W^{1,q}(M)$ for any $1<q<2$ and in $C^2_{\text{loc}}(M\setminus\{p\})$, where $\overline{u_\e}=\int_Mu_\e dv_g$
and $G_p$ is the Green function which satisfies
\begin{align}\label{green}
\begin{cases}
\Delta G_p=8\pi-8\pi\delta_p,\\
\int_M G_p dv_g=0.
\end{cases}
\end{align}

In a normal coordinate system around $y$, one has
\begin{align}\label{ex-green}
G_y(x)=-4\log r+A_y+b_1x_1+b_2x_2+c_1x_1^2+2c_2x_1x_2+c_3x_2^2+O(r^3),
\end{align}
where $r=\text{dist}(x,y)$.

\textbf{Fact 3.} In $M\setminus B_{Rr_\e}(x_\e)$, $u_\e\geq G_{x_\e}-\lambda_\e-2\log(\pi h(p))-A_p+o_R(1)+o_\e(1)$.

We remark that Fact 3 is based on the maximum principle.

With these estimates, they proved for positive $h$ that
\begin{align}\label{bound}
\inf_{W^{1,2}(M)}J(u)\geq C_0=-8\pi-8\pi\log\pi-4\pi\max_{M}(A_p+2\log h(p)).
\end{align}
Then they constructed a sequence $\phi_\e\in W^{1,2}(M)$ such that $J(\phi_\e)<C_0$ for sufficiently small $\e>0$, provided that
\begin{align}\label{djlw-cond}
\Delta \log h(p_0)+8\pi-2K(p_0)>0,
\end{align}
where $p_0$ is the maximum point of $A_p+2\log h(p)$ on $M$ and $K$ is the Gauss curvature of $(M,g)$.
We call (\ref{djlw-cond}) as the Ding-Jost-Li-Wang condition. When $h\geq0$ and is positive somewhere,
Yang and the author \cite{YZ} proved that the blow-up can not happen on zero points of $h$, and showed the Ding-Jost-Li-Wang condition is sufficient to solve (\ref{kw-eq}).
The proof is based on a concentration lemma and some elliptic estimates.

Later, using the flow introduced by Cast\'{e}ras \cite{C1,C2}, Li and Zhu \cite{LZ} and Sun and Zhu \cite{SZ} reproved the existence result in \cite{DJLW97} and \cite{YZ} respectively,  Wang and Yang \cite{WY} showed on a symmetric Riemann surface, equations like (\ref{kw-eq}) can be solved under (\ref{djlw-cond}) when $h$ changes sign.
Recently, using another
flow, Li and Xu \cite{LX} proved that when $h$ changes sign, the Ding-Jost-Li-Wang condition is still sufficient to solve (\ref{kw-eq}).

In this paper, we would like to give a direct proof which is based on the variational method and the maximum principle to the following result.

\begin{theorem}\label{thm}
Let $(M,g)$ be a compact Riemann surface, $h$ a smooth function on $M$ which is positive somewhere. Denote $M_+=\{x\in M: h(x)>0\}$. Suppose the Ding-Jost-Li-Wang condition
(\ref{djlw-cond}) is satisfied on $M_+$. Then the functional $J$ attains its minimum in $H_1$ and consequently the Kazdan-Warner equation (\ref{kw-eq}) has a smooth solution.
\end{theorem}

\begin{remark}
In \cite{SZ2}, the authors first asserted Theorem \ref{thm}. They claimed without much more details that if $u_\e$ blows up, it must happen at a single point $p$ and $h(p)>0$. Also, they said that the maximum principle does not work since $h$ changes sign. Instead, they used the capacity estimate to deal with the energy on the neck domain.
\end{remark}

The aim of this paper is two fold. First, we give the details on
the claim that if $u_\e$ blows up, it must happen at a single point $p$ and $h(p)>0$.
Second, we find that the maximum principle still works after some modified arguments, so we could still use the maximum principle to deal with the energy on the neck domain like in \cite{DJLW97}. Our method also works for the equation with singular sources, we will present this in a forthcoming paper.

We organized this paper as follows: In Sect. 2, we shall derive an explicit lower bound of $J$ in $H_1$ when $u_\e$ blows up;
In Sect. 3, we prove the main theorem. Throughout this paper, sequence and its subsequence will
not be distinguished and the constant $C$ may change from line to line.

\section{The lower bound}

In this section, we shall derive an explicit lower bound of $J$ in $H_1$ when $u_\e$ blows up,
so we assume $u_\e$ blows up throughout this section. First of all, we have
\begin{lemma}\label{int}
There exist positive constants $C_1$ and $C_2$ such that $$C_1\leq \int_M e^{u_\e}dv_g\leq C_2.$$
\end{lemma}
\begin{proof}
Since $u_\e\in H_1$, one has $\int_M he^{u_\e}dv_g=1$ and then
\begin{align*}
\frac{1}{\max_M h}\leq\int_M e^{u_\e}dv_g.
\end{align*}
Since $H_1\neq\emptyset$, we choose $u_0\in H_1$, then $J_\e(u_\e)=\inf_{H_1}J_\e(u)\leq J_\e(u_0)\to J(u_0)\leq C$, then the Moser-Trudinger inequality (\ref{MT}) and Jensen's inequality yield
\begin{align*}
\log\int_M e^{u_\e}dv_g&\leq\frac{1}{16\pi}\int_M|\nabla u_\e|^2dv_g+\overline{u_\e}+C\\
&=\frac{1}{8\pi}J_\e(u_\e)+\frac{\e}{8\pi}\overline{u_\e}+C\\
&\leq \frac{\e}{8\pi}\log\int_M e^{u_\e}dv_g+C.
\end{align*}
So we have $\int_M e^{u_\e}dv_g\leq C$.
\end{proof}

By Lemma \ref{int}, one knows that Lemma 2.3 in \cite{DJLW97} holds, i.e.
\begin{lemma}\label{L_q}
For any $1<q<2$, $\|\nabla u_\e\|_{q}\leq C_q$.
\end{lemma}

Concerning $u_\e$ blows up, we have the following equivalent characterizations.
\begin{lemma}\label{blowup-kehua}
There holds
\begin{align*}
\lambda_\e\ra+\infty\Leftrightarrow\|\nabla u_\e\|_2\ra+\infty\Leftrightarrow\overline{u_\e}\ra-\infty.
\end{align*}
\end{lemma}
\begin{proof}
Since $J_\e(u_\e)$ is bounded, the second equivalent relation is obvious. We put our attention on the first one.
If $\|\nabla u_\e\|_2\ra+\infty$,  then one has $\lambda_\e\ra+\infty$ by Lemma 2.4 in \cite{DJLW97}.
If $\|\nabla u_\e\|_2\leq C$, then $\overline{u_\e}$ is bounded. One has $e^{u_\e}$ is bounded in $L^q(M)$ for any $q\geq1$ by the Moser-Trudinger inequality.
Then by the standard elliptic estimates, we have $\|u_\e\|_{L^{\infty}(M)}$ is bounded. Therefore $\lambda_\e\leq C$. This ends the proof.
\end{proof}

\begin{remark}
We call $u_\e$ blows up, if one of the three items in Lemma \ref{blowup-kehua} holds.
\end{remark}

By Brezis-Merle's lemma (\cite{BM}, Theorem 1) and following elliptic estimates as the proof of Lemma 2.8 in \cite{DJLW97}, one has
\begin{lemma}\label{bm}
Let $\Omega\subset M$ be a domain. If for some $\gamma\in(0,1/2)$
$$\int_{\Omega}|h|e^{u_\e}dv_g<\frac12-\gamma,$$
then
$$\|u_\e-\overline{u_\e}\|_{L_{\text{loc}}^{\infty}(\Omega)}\leq C.$$
\end{lemma}

Due to Lemma \ref{bm}, we define the blowup set of $u_\e$ as
\begin{align}\label{blowup-set}
S=\{x\in M:~\lim_{r\ra0}\lim_{\e\ra0}\int_{B_r(x)}|h|e^{u_\e}dv_g\geq\frac{1}{2}\}.
\end{align}
The following lemma is very important to us. It asserts that $S$ is a single point set.
\begin{lemma}\label{blowup-set-lemma}
$S=\{p\}$.
\end{lemma}
\begin{proof}
We divide the proof into three parts.

1. $S\neq\emptyset$.

If not, $\forall x\in M$, $\exists\delta=\delta_x\in(0,1/2)$, $\exists r=r_x\in(0,\text{inj}(M))$ such that
\begin{align*}
\int_{B_{r_x}(x)}|h|e^{u_\e}dv_g<\frac12-\d.
\end{align*}
Then Lemma \ref{bm} tells us that
\begin{align*}
\|u_\e-\overline{u_\e}\|_{L^\infty(B_{r_x/2}(x))}\leq C.
\end{align*}
Combining with a finite covering argument, we arrive at
\begin{align*}
\|u_\e-\overline{u_\e}\|_{L^\infty(M)}\leq C.
\end{align*}
Multiplying equation (\ref{kw-eq1}) with $u_\e-\overline{u_\e}$ and integrating on $M$, integrating by parts
\begin{align*}
\int_M|\nabla u_\e|^2dv_g\leq C\|h\|_{L^\infty(M)}\int_{M}e^{u_\e}dv_g\leq C,
\end{align*}
where in the last inequality we have used Lemma \ref{int}. This contradiction with $\|\na u_\e\|_2\ra+\infty$ tells us that $S$ is not empty.

2. $\sharp S=1$.

If not, since $S\neq\emptyset$, we have $\sharp S\geq2$. Choosing $x_1\neq x_2\in S$, one has for sufficiently small $r$ that
\begin{align*}
\frac{\int_{B_r(x_1)}e^{u_\e}dv_g}{\int_Me^{u_\e}dv_g}\geq\frac{1}{2C_2|h|_{\text{max}}},~~
\frac{\int_{B_r(x_2)}e^{u_\e}dv_g}{\int_Me^{u_\e}dv_g}\geq\frac{1}{2C_2|h|_{\text{max}}}.
\end{align*}
Then by the distribution lemma of Chen-Li (\cite{CL2}, Theorem 2.1), we obtain
\begin{align*}
\log\int_M e^{u_\e}dv_g\leq (\frac{1}{32\pi}+\e')\int_{M}|\nabla u_\e|^2dv_g+\overline{u_\e}+C.
\end{align*}
Hence
\begin{align*}
C\geq J_\e(u_\e)&=\frac12\int_{M}|\nabla u_\e|^2 dv_g+(8\pi-\e)\overline{u_\e}-(8\pi-\e)\log\int_{M}he^{u_\e}dv_g\nonumber\\
&\geq \frac12\int_{M}|\nabla u_\e|^2dv_g-(8\pi-\e)\log h_{\text{max}}-(8\pi-\e)(\frac{1}{32\pi}+\e')\int_{M}|\nabla u_\e|^2dv_g\nonumber\\
&\geq\frac16\int_M |\nabla u_\e|^2dv_g-C.
\end{align*}
It follows $\int_M |\nabla u_\e|^2dv_g\leq C$, this contradiction tells us that $\sharp S=1$.

3. $S=\{p\}$.

Recall that $u_\e(x_\e)=\max_{M}u_\e$, and $x_\e\ra p\in M$ as $\e\ra0$.

If $p\notin S$, then $\exists\delta_p>0$ and $r_p\in(0,\text{inj}(M))$ such that
 $$\int_{B_{r_p}(p)}|h|e^{u_\e}dv_g<\frac12-\delta_p.$$
Then by Lemma \ref{bm}, one has
\begin{align*}
\|u_\e-\overline{u_\e}\|_{L^{\infty}(B_{r_p/2}(p))}\leq C.
\end{align*}
Hence
\begin{align*}
u_\e(x_\e)\leq C+\overline{u_\e}\leq C,
\end{align*}
where in the last inequality we have used Jensen's inequality and Lemma \ref{int}.
On the other hand, for $x\in S$ and small $r$
\begin{align*}
u_\e(x_\e)=\max_M u_\e\geq\max_{B_r(x)}u_\e\ra+\infty~~\text{as}~~\e\ra0.
\end{align*}
This contradiction tells us that $p\in S$, concluding we know $S=\{p\}$.
\end{proof}

Still, we can show that at the blow-up point, $h$ is positive. Namely,

\begin{lemma}\label{h-positive}
$h(p)>0$.
\end{lemma}
\begin{proof}
By the definition of $S$ and Lemma \ref{int}, one has $h(p)\neq0$. Now we just need to exclude $h(p)$ is negative. We do this by the maximum principle.
If $h(p)<0$, then $h(x_\e)<0$ for sufficiently small $\e$. Therefore, at $x_\e$
\begin{align*}
\Delta u_\e(x_\e)=(8\pi-\e)(1-h(x_\e)e^{u_\e(x_\e)})>0.
\end{align*}
This contradicts with the fact that $u_\e$ attains its maximum at $x_\e$. So we have $h(p)>0$.
\end{proof}

We choose a local normal coordinate system around $p$, denote $r_\e=e^{-\lambda_\e/2}$ and define
\begin{align*}
\varphi_\e(x)=u_\e(x_\e+r_\e x)-\lambda_\e.
\end{align*}
Since $h(p)>0$, the proof of Lemma 2.5 in \cite{DJLW97} tells us that
\begin{align}\label{bubble}
\varphi_\e(x)\ra\varphi_0(x)=-2\log(1+\pi h(p)|x|^2)~~\text{in}~~C_{\text{loc}}^1(\mathbb{R}^2)~~\text{as}~~\e\ra0.
\end{align}
Calculating directly, one has
\begin{align}\label{int=1}
\lim_{R\ra+\infty}\lim_{\e\ra0}\int_{B_{Rr_\e}(x_\e)}he^{u_\e}dv_g=\int_{\mathbb{R}^2}h(p)e^{\varphi_0}dx=1.
\end{align}

Since $S=\{p\}$, for any $x\in M\setminus\{p\}$, there exists a $\gamma_x\in(0,1/2)$ and a small $r_x\in(0,\text{dist}(x,p)/2)$ such that
$$\int_{B_{r_x}(x)}|h|e^{u_\e}dv_g<\frac12-\gamma_x.$$
By Lemma \ref{bm}, $\|u_\e-\overline{u_\e}\|_{L^{\infty}_{B_{r_x/2}(x)}}\leq C$, then by Lemma \ref{blowup-kehua} we have
$u_\e(x)\leq C+\overline{u_\e}\ra-\infty$ as $\e\ra0$. So for any $\Omega\subset\subset M\setminus\{p\}$, there holds
\begin{align}\label{int=0}
\int_{\Omega}|h|e^{u_\e}dv_g\ra0~~\text{as}~~\e\ra0.
\end{align}

By (\ref{int=1}) and (\ref{int=0}) we can see that, $he^{u_\e}$ converges to $\delta_p$ in the sense of measure. Therefore $u_\e-\overline{u_\e}\ra G_p(x)$ weakly in $W^{1,q}(M)$ for any $1<q<2$, where $G_p$ is the Green function satisfying (\ref{green}), since $G_p$ is the only solution of (\ref{green}) in $W^{1,q}(M)$. Lemma \ref{bm}
and (\ref{int=0}) yield that for any $\Omega\subset\subset M\setminus\{p\}$,
\begin{align*}
\|u_\e-\overline{u_\e}\|_{L^{\infty}(\Omega)}\leq C.
\end{align*}
This inequality together with the standard elliptic estimates yields that
\begin{align}\label{convergence}
u_\e-\overline{u_\e}\ra G_p~~\text{in}~~C_{\text{loc}}^{2}(M\setminus\{p\})~~\text{as}~~\e\ra0.
\end{align}

\begin{remark}
Even with a lot of effort, (\ref{bubble}) and (\ref{convergence}) show Fact 1 and Fact 2 in the introduction are still right when $h$ changes sign.
However, Fact 3 in the introduction is out of reach since $h$ changes sign and the maximum principle does not work on the whole $M\setminus B_{Rr_\e}(x_\e)$.
\end{remark}

However, since $h(p)>0$, we can still use the maximum principle on a small neck neighborhood of $x_\e$ to build the estimate like Fact 3. Fortunately, it is enough to estimate the energy on the neck domain.

First, we need to show
\begin{lemma}\label{maxi}
There exists a constant $C_3$ such that
\begin{align*}
\lambda_\e+\overline{u_\e}\geq -C_3.
\end{align*}
\end{lemma}
\begin{proof}
Multiplying (\ref{kw-eq1}) with $u_\e$ and integrating on $M$, one has
\begin{align}\label{maxi-1}
-\int_M |\nabla u_\e|^2dv_g = (8\pi-\e)\overline{u_\e}- (8\pi-\e)\int_{M}he^{u_\e}u_\e dv_g
\end{align}
Since $u_\e\in H_1$, one has by (\ref{maxi-1}) that
\begin{align}\label{maxi-2}
2J_\e(u_\e)=&(8\pi-\e)\int_Mhe^{u_\e}(u_\e+\overline{u_\e})dv_g.
\end{align}
It follows from Lemma \ref{bm} that $\|u_\e-\overline{u_\e}\|_{L^{\infty}(M\setminus B_\delta(p))}\leq C$, from Lemma \ref{blowup-kehua} that $\overline{u_\e}\ra-\infty$, so
\begin{align}\label{maxi-3}
&\int_{M\setminus B_\delta(p)}he^{u_\e}(u_\e+\overline{u_\e})dv_g\nonumber\\
=&\int_{M\setminus B_{\delta}(p)}he^{u_\e-\overline{u_\e}}e^{\overline{u_\e}}(u_\e-\overline{u_\e})dv_g
  +2\int_{M\setminus B_{\delta}(p)}he^{u_\e-\overline{u_\e}}e^{\overline{u_\e}}\overline{u_\e} dv_g\nonumber\\
=&o_\e(1)~~\text{as}~~\e\ra0.
\end{align}
By (\ref{int=0}),
$$\int_{M\setminus B_\delta(p)}he^{u_\e}dv_g=o_\e(1)~~\text{as}~~\e\ra0.$$
This together with $u_\e\in H_1$ tells us that
\begin{align}\label{maxi-4}
\int_{B_\delta(p)}he^{u_\e}dv_g=1+o_\e(1)~~\text{as}~~\e\ra0.
\end{align}
  Choose $\delta>0$ small such that $h>0$ in $B_\delta(p)$, one obtains by (\ref{maxi-4}) that
\begin{align}\label{maxi-5}
\int_{B_\delta(p)}he^{u_\e}(u_\e+\overline{u_\e})dv_g\leq (\lambda_\e+\overline{u_\e})\int_{B_\delta(p)}he^{u_\e}=(\lambda_\e+\overline{u_\e})(1+o_\e(1))~~\text{as}~~\e\ra0.
\end{align}
Taking (\ref{maxi-3}) and (\ref{maxi-5}) into (\ref{maxi-2}), we complete the proof by noticing $J_\e(u_\e)$ is bounded.
\end{proof}

In the following, we let $\delta>0$ small enough such that $h>0$  and (\ref{ex-green}) holds for $G_{x_\e}$ in $B_{\delta}(x_\e)$.
Recall that $r_\e=e^{-\lambda_\e/2}$ and $R>0$. Due to the maximum principle, we have the following estimate on the neck domain.
\begin{lemma}\label{maximum}
There exists a constant $C_4$, such that in $B_\delta(x_\e)\setminus B_{Rr_\e}(x_\e)$
\begin{align*}
u_\e\geq G_{x_\e}-\lambda_\e-C_4+o_\e(1)+o_R(1).
\end{align*}
%when $\e>0$ is small enough and $R>0$ is big enough.
\end{lemma}
\begin{proof}
It is clear that for any constant $C$
\begin{align}\label{maximum-0}
\Delta(u_\e-G_{x_\e}+\lambda_\e-C)=-\e-(8\pi-\e)he^{u_\e}<0~~\text{in}~~B_\delta(x_\e)\setminus B_{Rr_\e}(x_\e).
\end{align}
By (\ref{bubble}) and (\ref{ex-green}) one has on $\p B_{Rr_\e}(x_\e)$ that
\begin{align*}
u_\e&=\lambda_\e-2\log(1+\pi h(p)R^2)+o_\e(1),\\
G_{x_\e}&=-4\log(Rr_\e)+A_{x_\e}+o_\e(1).
\end{align*}
Then
\begin{align}\label{maximum-1}
u_\e=G_{x_\e}-\lambda_\e-2\log\frac{1+\pi h(p)R^2}{R^2}-A_{x_\e}+o_\e(1)~~\text{on}~~\p B_{Rr_\e}(x_\e).
\end{align}
By (\ref{convergence}) and (\ref{ex-green}) one has on $\p B_{\delta}(x_\e)$ that
\begin{align*}
u_\e&=\overline{u_\e}+G_p+o_\e(1),\\
G_{x_\e}&=G_p+o_\e(1).
\end{align*}
Then we have
$$u_\e=\overline{u_\e}+G_{x_\e}+o_\e(1)~~\text{on}~~\p B_\delta(x_\e).$$
This and Lemma \ref{maxi} yield
\begin{align}\label{maximum-2}
u_\e\geq G_{x_\e}-\lambda_\e-C_3+o_\e(1)~~\text{on}~~\p B_\delta(x_\e).
\end{align}
From (\ref{maximum-1}) and (\ref{maximum-2}), one can choose a large number $C_4$ such that for small enough $\e>0$ and large enough $R>0$
\begin{align}\label{maximum-3}
u_\e\geq G_{x_\e}-\lambda_\e-C_4~~\text{on}~~\p \left(B_\delta(x_\e)\setminus B_{Rr_\e}(x_\e)\right).
\end{align}
Combining (\ref{maximum-0}) and (\ref{maximum-3}), the lemma follows by the maximum principle.
\end{proof}

We still need the following lemma.
\begin{lemma}\label{bd-ab}
For any $\a>1$, there holds
\begin{align*}
\lambda_\e+\a\overline{u_\e}\leq C.
\end{align*}
\end{lemma}
\begin{proof}
Let $\a>1$ be fixed. By the Moser-Trudinger inequality (\ref{MT}) one has
\begin{align}\label{bd-1}
\log\int_{M}e^{\a u_\e}dv_g&\leq\frac{\a^2}{16\pi}\int_{M}|\nabla u_\e|^2dv_g+\a \overline{u_\e}+C\nonumber\\
&=\frac{\a^2}{8\pi}J_\e(u_\e)+(-\frac{8\pi-\e}{8\pi}\a^2+\a)\overline{u_\e}+C\nonumber\\
&\leq(-\a^2+\a)\overline{u_\e}+C,
\end{align}
where in the last inequality we have used facts that $J_\e(u_\e)$ is bounded and $\overline{u_\e}\ra-\infty$.

On the other hand, one has
\begin{align*}
\int_Me^{\a u_\e}dv_g\geq\int_{B_{Rr_\e}(x_\e)}e^{\a u_\e}dv_g&=e^{(\a-1)\lam_\e}\left(\int_{\mathbb{R}^2}e^{\a \varphi_0}dx+o_\e(1)+o_R(1)\right)\\
&=e^{(\a-1)\lam_\e}\left(\frac{1}{(2\a-1)h(p)}+o_\e(1)+o_R(1)\right).
\end{align*}
So we have
\begin{align}\label{bd-2}
\log\int_Me^{\a u_\e}dv_g\geq (\a-1)\lam_\e+C.
\end{align}
Then the lemma follows by (\ref{bd-1}) and (\ref{bd-2}).
\end{proof}

Now we are prepared to derive a lower bound of $J_\e(u_\e)$ when $u_\e$ blows up.

First, it follows from (\ref{bubble}) that
\begin{align}\label{in}
\int_{B_{Rr_\e}(x_\e)}|\nabla u_\e|^2dv_g&=\int_{\mathbb{B}_{R}(0)}|\nabla_{\mathbb{R}^2}\varphi_0|^2dx+o_\e(1)\nonumber\\
&=16\pi\log(1+\pi h(p)R^2)-16\pi+o_\e(1)+o_R(1).
\end{align}

Second, (\ref{convergence}) and (\ref{ex-green}) yield
\begin{align}\label{out}
\int_{M\setminus B_{\delta}(x_\e)}|\nabla u_\e|^2dv_g&=\int_{M\setminus B_\delta(x_\e)}|\nabla (u_\e-\overline{u_\e})|^2dv_g\nonumber\\
&=\int_{M\setminus B_\delta(p)}|\nabla G_p|^2dv_g+o_\e(1)\nonumber\\
&=-\int_{\p B_\delta(p)}G_p\frac{\p G_p}{\p n}ds_g+o_\e(1).
%&=-32\pi\log\delta+8\pi A_p+o_\e(1)+o_\delta(1).
\end{align}

Last, we estimate $\int_{B_\delta(x_\e)\setminus B_{Rr_\e}(x_\e)}|\nabla u_\e|^2dv_g$.
Since $u_\e$ satisfies equation (\ref{kw-eq1}), one has
\begin{align}\label{neck-1}
\int_{B_\delta(x_\e)\setminus B_{Rr_\e}(x_\e)}|\nabla u_\e|^2dv_g
=&-(8\pi-\e)\int_{B_\delta(x_\e)\setminus B_{Rr_\e}(x_\e)}u_\e dv_g\nonumber\\
 &+(8\pi-\e)\int_{B_\delta(x_\e)\setminus B_{Rr_\e}(x_\e)}he^{u_\e}u_\e dv_g\nonumber\\
 &+\int_{\p B_\delta(x_\e)}u_\e\frac{\p u_\e}{\p n}ds_g-\int_{\p B_{Rr_\e}(x_\e)}u_\e\frac{\p u_\e}{\p n}ds_g.
\end{align}
Using the fact $u_\e\in H_1$, (\ref{int=1}) and (\ref{int=0}), we have
\begin{align*}
\int_{B_\delta(x_\e)\setminus B_{Rr_\e}(x_\e)} he^{u_\e}dv_g=o_\e(1)+o_R(1).
\end{align*}
Then it follows from Lemma \ref{maximum} that
\begin{align}\label{neck-2}
(8\pi-\e)\int_{B_\delta(x_\e)\setminus B_{Rr_\e}(x_\e)}he^{u_\e}u_\e dv_g
\geq&(8\pi-\e)\int_{B_\delta(x_\e)\setminus B_{Rr_\e}(x_\e)}he^{u_\e}G_{x_\e}dv_g\nonumber\\
&-(8\pi-\e)\lambda_\e\int_{B_\delta(x_\e)\setminus B_{Rr_\e}(x_\e)}he^{u_\e}dv_g+o_\e(1)+o_R(1).
\end{align}
By equation (\ref{kw-eq1}) and the Green formula, one has that
\begin{align*}
(8\pi-\e)\int_{B_\delta(x_\e)\setminus B_{Rr_\e}(x_\e)}he^{u_\e}G_{x_\e}dv_g
=&-8\pi\int_{B_\delta(x_\e)\setminus B_{Rr_\e}(x_\e)}u_\e dv_g-\int_{\p B_\delta(x_\e)}G_{x_\e}\frac{\p u_\e}{\p n}ds_g\\
 &+\int_{\p B_\delta(x_\e)}u_\e\frac{\p G_{x_\e}}{\p n}ds_g+\int_{\p B_{Rr_\e}(x_\e)}G_{x_\e}\frac{\p u_\e}{\p n}ds_g\\
 &-\int_{\p B_{Rr_\e}(x_\e)}u_\e\frac{\p G_{x_\e}}{\p n}ds_g+(8\pi-\e)\int_{B_\delta(x_\e)\setminus B_{Rr_\e}(x_\e)}G_{x_\e}dv_g.
\end{align*}
Equation (\ref{kw-eq1}) also yields
\begin{align*}
-(8\pi-\e)\lambda_\e\int_{B_\delta(x_\e)\setminus B_{Rr_\e}(x_\e)}he^{u_\e}dv_g
=&-(8\pi-\e)\text{Vol}\left(B_\delta(x_\e)\setminus B_{Rr_\e}(x_\e)\right)\lam_\e\\
 &+\lam_\e\int_{\p B_{\delta}(x_\e)}\frac{\p u_\e}{\p n}ds_g-\lam_\e\int_{\p B_{Rr_\e}(x_\e)}\frac{\p u_\e}{\p n}ds_g.
\end{align*}
Taking these two estimates into (\ref{neck-2}) first and then putting (\ref{neck-2}) into (\ref{neck-1}), we have
\begin{align}\label{neck-3}
\int_{B_\delta(x_\e)\setminus B_{Rr_\e}(x_\e)}|\nabla u_\e|^2dv_g
=&-(16\pi-\e)\int_{B_\delta(x_\e)\setminus B_{Rr_\e}(x_\e)}u_\e dv_g-\int_{\p B_\delta(x_\e)}G_{x_\e}\frac{\p u_\e}{\p n}ds_g\nonumber\\
 &+\int_{\p B_\delta(x_\e)}u_\e\frac{\p G_{x_\e}}{\p n}ds_g+\int_{\p B_{Rr_\e}(x_\e)}G_{x_\e}\frac{\p u_\e}{\p n}ds_g\nonumber\\
 &-\int_{\p B_{Rr_\e}(x_\e)}u_\e\frac{\p G_{x_\e}}{\p n}ds_g+(8\pi-\e)\int_{B_\delta(x_\e)\setminus B_{Rr_\e}(x_\e)}G_{x_\e}dv_g\nonumber\\
 &+\lam_\e\int_{\p B_{\delta}(x_\e)}\frac{\p u_\e}{\p n}ds_g-\lam_\e\int_{\p B_{Rr_\e}(x_\e)}\frac{\p u_\e}{\p n}ds_g\nonumber\\
 &+\int_{\p B_\delta(x_\e)}u_\e\frac{\p u_\e}{\p n}ds_g-\int_{\p B_{Rr_\e}(x_\e)}u_\e\frac{\p u_\e}{\p n}ds_g\nonumber\\
 &-(8\pi-\e)\text{Vol}\left(B_\delta(x_\e)\setminus B_{Rr_\e}(x_\e)\right)\lam_\e+o_\e(1)+o_{\delta}(1).
\end{align}

By (\ref{convergence}) one has
\begin{align}\label{neck-31}
-\int_{\p B_\delta(x_\e)}G_{x_\e}\frac{\p u_\e}{\p n}ds_g=-\int_{\p B_\d(p)}G_p\frac{\p G_p}{\p n}ds_g+o_\e(1)
\end{align}
and
\begin{align}\label{neck-32}
\int_{\p B_\delta(x_\e)}u_\e\frac{\p G_{x_\e}}{\p n}ds_g=&\int_{\p B_\delta(x_\e)}(G_p+\overline{u_\e}+o_\e(1))\frac{\p G_{x_\e}}{\p n}ds_g\nonumber\\
=&\int_{\p B_\delta(x_\e)}G_p\frac{\p G_{x_\e}}{\p n}ds_g+\overline{u_\e}\int_{\p B_\delta(x_\e)}\frac{\p G_{x_\e}}{\p n}ds_g+o_\e(1)\nonumber\\
=&\int_{\p B_\delta(p)}G_p\frac{\p G_{p}}{\p n}ds_g-8\pi(1-\text{Vol}(B_\d(x_\e)))\overline{u_\e}+o_\e(1).
\end{align}
Using (\ref{kw-eq1}), (\ref{convergence}) and Lemma \ref{bd-ab} we have
\begin{align}\label{neck-33}
\lam_\e\int_{\p B_{\delta}(x_\e)}\frac{\p u_\e}{\p n}ds_g=&-\lam_\e\int_{M\setminus B_\d(x_\e)}\Delta u_\e dv_g\nonumber\\
=&-(8\pi-\e)(1-\text{Vol}(B_\d(x_\e)))\lam_\e+(8\pi-\e)\lam_\e\int_{M\setminus B_\d(x_\e)}he^{u_\e}dv_g\nonumber\\
=&-(8\pi-\e)(1-\text{Vol}(B_\d(x_\e)))\lam_\e+(8\pi-\e)\lam_\e e^{\overline{u_\e}}\int_{M\setminus B_\d(x_\e)}he^{u_\e-\overline{u_\e}}dv_g\nonumber\\
=&-(8\pi-\e)(1-\text{Vol}(B_\d(x_\e)))\lam_\e+o_\e(1)
\end{align}
and
\begin{align}\label{neck-34}
\int_{\p B_\d(x_\e)}u_\e\frac{\p u_\e}{\p n}ds_g=&\int_{\p B_\d(x_\e)}(G_p+\overline{u_\e}+o_\e(1))\frac{\p u_\e}{\p n}ds_g\nonumber\\
=&\int_{\p B_\d(p)}G_p\frac{\p G_p}{\p n}ds_g+\overline{u_\e}\int_{\p B_\d(x_\e)}\frac{\p u_\e}{\p n}ds_g+o_\e(1)\nonumber\\
=&\int_{\p B_\d(p)}G_p\frac{\p G_p}{\p n}ds_g-(8\pi-\e)(1-\text{Vol}(B_\d(x_\e)))\overline{u_\e}\nonumber\\
&+(8\pi-\e)\overline{u_\e} e^{\overline{u_\e}}\int_{M\setminus B_\d(x_\e)}he^{u_\e-\overline{u_\e}}dv_g+o_\e(1)\nonumber\\
=&\int_{\p B_\d(p)}G_p\frac{\p G_p}{\p n}ds_g-(8\pi-\e)(1-\text{Vol}(B_\d(x_\e)))\overline{u_\e}+o_\e(1).
\end{align}

By (\ref{maximum-1}) and (\ref{bubble}) one obtains
\begin{align}\label{neck-35}
&-\int_{\p B_{Rr_\e}(x_\e)}\frac{\p u_\e}{\p n}(u_\e-G_{x_\e}+\lambda_\e)ds_g\nonumber\\
=&\frac{8\pi^2h(p)R^2}{1+\pi h(p)R^2}(-A_{x_\e}-2\log\frac{1+\pi h(p)R^2}{R^2})+o_\e(1)+o_R(1)\nonumber\\
=&-8\pi A_p-16\pi\log\pi-16\pi\log h(p)+o_\e(1)+o_R(1).
\end{align}
From (\ref{convergence}) we know
\begin{align}\label{neck-36}
-\int_{\p B_{Rr_\e}(x_\e)}u_\e\frac{\p G_{x_\e}}{\p n}ds_g
=&-\lam_\e\int_{\p B_{Rr_\e}(x_\e)}\frac{\p G_{x_\e}}{\p n}ds_g-16\pi\log(1+\pi h(p)R^2)\nonumber\\
          &+o_\e(1)+o_R(1)\nonumber\\
=&8\pi(1-\text{Vol}(B_{Rr_\e}(x_\e)))\lam_\e-16\pi\log(1+\pi h(p)R^2)\nonumber\\
&+o_\e(1)+o_R(1).
\end{align}

One has by Lemma \ref{L_q} that
\begin{align}\label{neck-37}
\int_{B_\delta(x_\e)\setminus B_{Rr_\e}(x_\e)}u_\e dv_g
=&\int_{B_\delta(x_\e)\setminus B_{Rr_\e}(x_\e)}(u_\e-\overline{u_\e})dv_g+\text{Vol}(B_\delta(x_\e)\setminus B_{Rr_\e}(x_\e))\overline{u_\e}\nonumber\\
=&\text{Vol}(B_\delta(x_\e)\setminus B_{Rr_\e}(x_\e))\overline{u_\e}+o_\d(1).
\end{align}
It is clear that
\begin{align}\label{neck-38}
(8\pi-\e)\int_{B_\delta(x_\e)\setminus B_{Rr_\e}(x_\e)}G_{x_\e}dv_g=o_\e(1)+o_\delta(1),~~\lambda_\e \text{Vol}(B_{Rr_\e}(x_\e))=o_\e(1).
\end{align}
By Lemma \ref{bd-ab} one also has
\begin{align}\label{neck-39}
-\overline{u_\e}\text{Vol}(B_{Rr_\e}(x_\e))=o_\e(1).
\end{align}
Taking (\ref{neck-31})-(\ref{neck-39}) into (\ref{neck-3}), we have
\begin{align}\label{neck-last}
\int_{B_\delta(x_\e)\setminus B_{Rr_\e}(x_\e)}|\nabla u_\e|^2dv_g
\geq&\e\lam_\e-(16\pi-\e)\overline{u_\e}-16\pi\log(1+\pi h(p)R^2)\nonumber\\
&+\int_{B_\delta(p)}G_p\frac{\p G_p}{\p n}ds_g-8\pi A_p-16\pi\log\pi-16\pi\log h(p)\nonumber\\
&+o_\e(1)+o_R(1)+o_\d(1).
\end{align}

Finally, by (\ref{in}), (\ref{out}) and (\ref{neck-last}) we have
\begin{align*}
\int_{M}|\nabla u_\e|^2dv_g\geq&\e\lam_\e-(16\pi-\e)\overline{u_\e}-16\pi-8\pi A_p-16\pi\log\pi-16\pi\log h(p)\nonumber\\
&+o_\e(1)+o_R(1)+o_\d(1).
\end{align*}
Therefore,
\begin{align*}
J_\e(u_\e)=&\frac{1}{2}\int_M|\nabla u_\e|^2dv_g+(8\pi-\e)\overline{u_\e}\nonumber\\
\geq&\frac{\e}{2}(\lambda_\e-\overline{u_\e})-8\pi-4\pi A_p-8\pi\log\pi-8\pi\log h(p)+o_\e(1)+o_R(1)+o_\d(1)\nonumber\\
\geq&-8\pi-4\pi A_p-8\pi\log\pi-8\pi\log h(p)+o_\e(1)+o_R(1)+o_\d(1).
\end{align*}
Letting $\e\ra0$ first, then $R\ra+\infty$ and then $\d\ra0$, we have
\begin{align*}
\lim_{\e\ra0}\inf_{H_1}J_\e(u)\geq-8\pi-8\pi\log\pi-4\pi\max_{p\in M_+}(A_p+2\log h(p)).
\end{align*}
By this and Lemma 2.10 in \cite{DJLW97}, one has if $u_\e$ blows up
\begin{align}\label{low-bd}
\inf_{H_1}J(u)\geq-8\pi-8\pi\log\pi-4\pi\max_{p\in M_+}(A_p+2\log h(p)).
\end{align}

\section{Completion of the proof of Theorem \ref{thm}}

Let $\phi_\e$ be the blow up sequence constructed in \cite{DJLW97}.
If the Ding-Jost-Li-Wang condition (\ref{djlw-cond}) is satisfied on $M_+$,
we can repeat the proof in \cite{DJLW97} and show that for sufficiently small $\e$
\begin{align*}
J(\phi_\e)<-8\pi-8\pi\log\pi-4\pi\max_{p\in M_+}(A_p+2\log h(p)).
\end{align*}
It is easy to check that $\int_{M}he^{\phi_\e}dv_g>0$, and
$$\widetilde{\phi_\e}=\phi_\e-\log\int_{M}he^{\phi_\e}dv_g\in H_1.$$
Since $J(u+c)=J(u)$ for any constant $c$, we have for sufficiently small $\e$
\begin{align}\label{up-bd}
\inf_{H_1}J(u)\leq J(\widetilde{\phi_\e})=J(\phi_\e)<-8\pi-8\pi\log\pi-4\pi\max_{p\in M_+}(A_p+2\log h(p)).
\end{align}
It follows from (\ref{up-bd}) and (\ref{low-bd}) that $u_\e$ does not blow up. So $u_\e$ converges to some $u_0$ which minimizes $J$ in $H_1$ and solves
the Kazdan-Warner equation (\ref{kw-eq}). Standard elliptic estimates tell us $u_0$ is smooth. Finally, we finish the proof of Theorem \ref{thm}.\\

{\bf Acknowledgements} The author is supported by the National Science Foundation of China (Grant Nos. 11171347 and 11401575).
Part of this article was finished when the author visited School of Mathematics Science and China-France Mathematics Center at University of Science and Technology of China. He would like to thank them for their enthusiasm and the excellent working conditions they supplied for him.
 He thanks the referee for his/her careful review and very helpful suggestions.

{\bf Declarations} (Conflict of interest) The author declares there is no conflicts of interest.

\end{document}